\documentclass[a4paper]{article}
\usepackage[english]{babel}
\usepackage[utf8]{inputenc}
\usepackage{amsmath}
\usepackage{amsthm}
\usepackage{tikz}
\usepackage[margin=1.03125in]{geometry}
\usepackage{graphicx}

\title{Closed formulas for fractional chromatic polynomials of some common classes of graphs}
\author{Recuero, P. M.}
\date{}

\newtheorem{definition}{Definition}
\newtheorem{lemma}{Lemma}
\newtheorem{proposition}{Proposition}

\begin{document}
\maketitle
\begin{abstract}
Chromatic polynomials have been studied extensively, giving us results such as the Fundamental
Reduction Theorem and closed formulas for the chromatic polynomial of common classes of graphs.
Though, none of those extend to the context of fractional colorings. We thus present closed
formulas for the ``fractional" chromatic polynomial - a function that counts the number of distinct
$b$-fold $\lambda$-colorings on a given graph - of complete graphs, trees and forests, as well
as a generalized form of the Fundamental Reduction Theorem.
\end{abstract}

\section{Introduction}
A mapping $f: V_G \to [\lambda]$ is called a \textit{legal $\lambda$-coloring}
\footnote{Some authors refer to legal colorings merely as ``colorings", leaving it implicit
that the colorings in question are legal. For most of
this text we will also be doing so, save for a few exceptions, where we specify it out of convenience
or necessity.}
on the graph $G$ if for all adjacent vertices $u$ and $v$ of $G$, $f(u)\neq f(v)$.\\
Two colorings $f_1$ and $f_2$ on the graph $G$ are regarded as \textit{distinct} if
there exists some vertex $u$ in $G$ such that $f_1(u)\neq f_2(u)$.\\
The \textit{chromatic polynomial} of $G$, usually denoted as $\mathcal{P}(G,\lambda)$,
is the function that when inputted with positive integer $\lambda$, outputs the number
of distinct $\lambda$-colorings on the graph $G$. For instance, the function
\begin{equation}
\mathcal{P}(K_n,\lambda)=\prod_{j=0}^{n-1}(\lambda-j)
\end{equation}
tells us the number of distinct $\lambda$-colorings on the complete graph on $n$ vertices.\\
An important result regarding chromatic polynomials due to Whitney [3] is the following:
given a graph $G$, if $G+uv$ is the graph obtained from $G$ by adding an edge between two
non-adjacent vertices $u$ and $v$ of $G$, and $G/uv$ the graph obtained from $G$ by merging
those two vertices, then
\begin{equation*}
\mathcal{P}(G,\lambda)=\mathcal{P}(G+uv,\lambda)+\mathcal{P}(G/uv,\lambda)
\tag{FRT 1}
\end{equation*}
Similarly, if $G-uv$ is the graph obtained from $G$ by deleting the edge between two adjacent
vertices $u$ and $v$ of $G$ and $G/uv$ obtained by merging them, then
\begin{equation*}
\mathcal{P}(G,\lambda)=\mathcal{P}(G-uv,\lambda)-\mathcal{P}(G/uv,\lambda)
\tag{FRT 2}
\end{equation*}
Together, these recurrence relations are known as the \textit{Fundamental Reduction Theorem}, and
provide us with a recursive way of computing the chromatic polynomial of a given graph.
\medskip\\
\textit{Fractional} coloring is a generalization of the ``ordinary" coloring described above
where instead of each vertex being assigned a single color, each vertex is assigned a set - possibly
containing more than one element - of colors. See Figure 1.1 for an example.
\begin{samepage}
\begin{center}
\begin{tikzpicture}
\draw[black, very thick] (0,1) -- (0.9512,0.309);
\draw[black, very thick] (0,1) -- (-0.9512,0.309);
\draw[black, very thick] (0.5878,-0.809) -- (0.9512,0.309);
\draw[black, very thick] (-0.5878,-0.809) -- (-0.9512,0.309);
\draw[black, very thick] (0.5878,-0.809) -- (-0.5878,-0.809);

\begin{scope}
\clip (0.9512,0.309) circle (0.25);
                \fill[blue!50] (0.9512,0.059) rectangle (0.7012,0.559);
                \fill[yellow!50] (0.9512,0.059) rectangle (1.2012,0.559);
\end{scope}
\draw[black, very thick] (0.9512,0.309) circle (0.25);

\begin{scope}
\clip (-0.9512,0.309) circle (0.25);
                \fill[blue!50] (-0.9512,0.059) rectangle (-1.2012,0.559);
                \fill[pink!50] (-0.9512,0.059) rectangle (-0.7012,0.559);
\end{scope}
\draw[black, very thick] (-0.9512,0.309) circle (0.25);

\begin{scope}
\clip (0.5878,-0.809) circle (0.25);
                \fill[red!50] (0.5878,-1.059) rectangle (0.3378,-0.559);
                \fill[pink!50] (0.5878,-1.059) rectangle (0.8378,-0.559);
\end{scope}
\draw[black, very thick] (0.5878,-0.809) circle (0.25);

\begin{scope}
\clip (-0.5878,-0.809) circle (0.25);
                \fill[green!50] (-0.5878,-1.059) rectangle (-0.8378,-0.559);
                \fill[yellow!50] (-0.5878,-1.059) rectangle (-0.3378,-0.559);
\end{scope}
\draw[black, very thick] (-0.5878,-0.809) circle (0.25);

\begin{scope}
\clip (0,1) circle (0.25);
                \fill[green!50] (0,0.75) rectangle (-0.25,1.25);
                \fill[red!50] (0,0.75) rectangle (0.25,1.25);
\end{scope}
\draw[black, very thick] (0,1) circle (0.25);

\end{tikzpicture}
\end{center}
\begin{center}
\textit{Fig. 1.1}
\end{center}
\end{samepage}
\pagebreak
\begin{samepage}
We say then that a mapping $f:V_G \to [\lambda]^b$ is a (legal) $b$-fold $\lambda$-coloring on $G$
if for all adjacent vertices $u$ and $v$ of $G$, $f(u) \cap f(v) = \emptyset$. Thus we denote
the coloring on Figure 1.1 above as a $2$-fold $5$-coloring. Accordingly, an ordinary $\lambda$-coloring
would be called a $1$-fold $\lambda$-coloring
\footnote{We will use ``1-fold" and ``ordinary" interchangeably - whichever is better suited for the moment.}.\\
Informally, the requirements of fractional colorings are the same as the ones of ordinary colorings.
Though, in the context of fractional colorings, some requirements merely implicit or vacuous in ordinary
colorings become explicit, namely, (i) that all colors used to paint the same vertex be different from
one another, (ii) that adjacent vertices have no colors in common whatsoever and (iii) that all vertices
be painted with the same numer of colors.
\end{samepage}
\medskip\\
While closed formulas for the chromatic polynomials of the most common classes of graphs,
such as complete (equation (1) shown above) or cycle graphs, are known, they are restricted to ordinary colorings.
And likewise, so are the recurrence relations in the Fundamental Reduction Theorem. What we then present are generalizations
of the known formulas for the chromatic polynomials of complete graphs, trees and forests, and of the
recurrence relations in the Fundamental Reduction Theorem as well.

\section{Preliminaries
\protect\footnote{We begin by introducing a graph operation that shows us how the number of $b$-fold $\lambda$-colorings
on a given graph relates to the number of $1$-fold $\lambda$-colorings - which we can count - of another
graph. We first introduce the concept informally, as to motivate it, and then provide the due formalizations.}}
Suppose we are handed the graph $H$ consisting of two vertices, which we shall call $u$ and $v$, and
a single edge between them, and it is for us to decide whether or not $H$ is $3$-fold $\lambda$-colorable.
For that purpose, we first divide each vertex of $H$ in $3$ sections, each one of which
will be assigned a single color.
\begin{samepage}
\begin{center}
\begin{tikzpicture}
\filldraw[color=black!100, fill=black!0, very thick](-0.75,0) circle (0.5);
\filldraw[color=black!100, fill=black!0, very thick](0.75, 0) circle (0.5);
\draw[black, very thick] (-0.25,0) -- (0.25,0);
\draw[black, very thick] (-0.317,-0.25) -- (-0.75,0);
\draw[black, very thick] (-1.18,-0.25) -- (-0.75,0);
\draw[black, very thick] (-0.75, 0.5) -- (-0.75,0);
\draw[black, very thick] (0.317, -0.25) -- (0.75,0);
\draw[black, very thick] (1.18, -0.25) -- (0.75,0);
\draw[black, very thick] (0.75, 0.5) -- (0.75, 0);
\node at (-1.3,0.9){$H$};
\node at (-0.5,0.125){\small $u_3$};
\node at (-1,0.125){\small $u_1$};
\node at (-0.75,-0.25){\small $u_2$};
\node at (0.5,0.125){\small $v_1$};
\node at (1,0.125){\small $v_3$};
\node at (0.75,-0.25){\small $v_2$};
\end{tikzpicture}
\end{center}
\begin{center}
\textit{Fig. 2.1}
\end{center}
\end{samepage}
As per the requirements of fractional colorings earlier stated, two sections of $H$ must have
different colors assigned to them if (i) they belong to the same vertex or (ii) they belong to
adjacent vertices. Well, but this is an ordinary graph coloring problem in disguise - each
section is a vertex and sections having one of the above properties are adjacent
vertices. As such, let us construct a new graph, $H'$: for each section in $H$, let there
be a corresponding vertex in $H'$, and for every pair of sections in $H$ belonging to the
same or to adjacent vertices, let their corresponding vertices in $H'$ be adjacent.
\begin{samepage}
\begin{center}
\begin{tikzpicture}
\draw[black, very thick] (-1.5,2) -- (-1.5,1);
\draw[black, very thick] (-1.5,2) -- (-0.5,0);
\draw[black, very thick] (-1.5,2) -- (0.5,0);
\draw[black, very thick] (-1.5,2) -- (1.5,1);
\draw[black, very thick] (-1.5,2) -- (1.5,2);
\draw[black, very thick] (-1.5,1) -- (-0.5,0);
\draw[black, very thick] (-1.5,1) -- (0.5,0);
\draw[black, very thick] (-1.5,1) -- (1.5,1);
\draw[black, very thick] (-1.5,1) -- (1.5,2);
\draw[black, very thick] (-0.5,0) -- (0.5,0);
\draw[black, very thick] (-0.5,0) -- (1.5,1);
\draw[black, very thick] (-0.5,0) -- (1.5,2);
\draw[black, very thick] (0.5,0) -- (1.5,1);
\draw[black, very thick] (0.5,0) -- (1.5,2);
\draw[black, very thick] (1.5,1) -- (1.5,2);

\filldraw[color=black!100, fill=black!0, very thick] (-1.5,2) circle (0.375);
\filldraw[color=black!100, fill=black!0, very thick] (-1.5,1) circle (0.375);
\filldraw[color=black!100, fill=black!0, very thick] (-0.5,0) circle (0.375);
\filldraw[color=black!100, fill=black!0, very thick] (0.5,0) circle (0.375);
\filldraw[color=black!100, fill=black!0, very thick] (1.5,1) circle (0.375);
\filldraw[color=black!100, fill=black!0, very thick] (1.5,2) circle (0.375);

\node at (-1.5,2){\small $u_1$};
\node at (-1.5,1){\small $v_1$};
\node at (-0.5,0){\small $u_2$};
\node at (0.5,0){\small $v_2$};
\node at (1.5,1){\small $v_3$};
\node at (1.5,2){\small $u_3$};
\node at (-2.25,2.5){$H'$};

\end{tikzpicture}
\end{center}
\begin{center}
\textit{Fig. 2.2}
\end{center}
\end{samepage}
Clearly then, for any $3$-fold $\lambda$-coloring $f$ on $H$ there exists a corresponding
$1$-fold $\lambda$-coloring $f'$ on $H'$ such that $f$ is legal iff $f'$ is legal: we merely
paint vertices of $H'$ with the same colors as their corresponding sections in $H$; as adjacent
vertices are always paired with sections belonging to the same or adjacent vertices - which if
$f$ is legal, will not be painted with the same color -, the resulting coloring on $H'$ must be
legal as well.
\medskip\\
Indeed, we can apply such operation on any graph, and for any $b$, essentially turning
a $b$-fold coloring into a $1$-fold coloring. More formally,
\pagebreak
\begin{samepage}
\begin{definition}
Given a graph $G$, let $G^b$ be the graph constructed from $G$ as follows:
\medskip\\
1. For every vertex $u$ in $G$, let there be vertices $u_1, \ldots, u_b$ in $G^b$, and let
these vertices form a complete subgraph;\\
2. For every pair of vertices $u$ and $v$ in $G$, if $u$ and $v$ are adjacent, then let vertices
$u_i$ and $v_j$ of $G^b$ be adjacent as well, for all positive integers $i, j \leq b$.
\end{definition}
\end{samepage}
Thus $H'$ becomes $H^3$.
\begin{lemma}
$\chi_b(G) = \chi(G^b)$
\end{lemma}
\begin{proof}
Say we are given a legal $b$-fold $\lambda$-coloring $f$ on $G$. From any such $f$ we can
construct a legal $1$-fold $\lambda$-coloring $f'$ on $G^b$ as follows: pair sections of $G$
with vertices of $G^b$ in a way such that pairs of adjacent vertices of $G^b$ are paired with
sections belonging to the same or to adjacent vertices of $G$
\footnote{Such a pairing is certainly possible - we defined $H'$ from $H$ earlier from it. Indeed,
the two structures are isomorphic.}. Now, color every vertex of $G^b$
with the same color as the section that it is paired with. If $f$ is legal, then all pairs
of sections belonging to the same or adjacent vertices on $G$ must be assigned different colors,
and since those sections are the ones that correspond with adjacent vertices of $G^b$, in $f'$,
all adjacent vertices of $G^b$ are colored with different colors, thus $f'$ must be legal.\\
The converse can be proved analogously. The fact that for every legal $b$-fold $\lambda$-coloring
on $G$ there exists a legal $1$-fold $\lambda$-coloring on $G^b$ and \textit{vice-versa} implies
that $G$ is $b$-fold $\lambda$-colorable iff $G^b$ is $1$-fold $\lambda$-colorable, and thus,
\textit{a fortiori}, $\chi_b(G) = \chi(G^b)$.
\end{proof}
Now, seeing as for every legal $b$-fold $\lambda$-coloring on $G$ there exists a corresponding
$1$-fold $\lambda$-coloring on $G^b$, one might be inclined to think that there are as many
distinct $b$-fold $\lambda$-colorings on $G$ as there are $1$-fold $\lambda$-colorings on $G^b$,
or at least wonder how the two numbers relate to each other. Let us then examine a few examples
of corresponding colorings on the graphs $H$ and $H'$ with which we worked earlier.
\begin{samepage}
\begin{center}
\begin{tikzpicture}
\draw[black, very thick] (-1.125,1.5) -- (-1.125,0.75);
\draw[black, very thick] (-1.125,1.5) -- (-0.375,0);
\draw[black, very thick] (-1.125,1.5) -- (0.375,0);
\draw[black, very thick] (-1.125,1.5) -- (1.125,0.75);
\draw[black, very thick] (-1.125,1.5) -- (1.125,1.5);
\draw[black, very thick] (-1.125,0.75) -- (-0.375,0);
\draw[black, very thick] (-1.125,0.75) -- (0.375,0);
\draw[black, very thick] (-1.125,0.75) -- (1.125,0.75);
\draw[black, very thick] (-1.125,0.75) -- (1.125,1.5);
\draw[black, very thick] (-0.375,0) -- (0.375,0);
\draw[black, very thick] (-0.375,0) -- (1.125,0.75);
\draw[black, very thick] (-0.375,0) -- (1.125,1.5);
\draw[black, very thick] (0.375,0) -- (1.125,0.75);
\draw[black, very thick] (0.375,0) -- (1.125,1.5);
\draw[black, very thick] (1.125,0.75) -- (1.125,1.5);

\filldraw[color=black!100, fill=red!50, very thick] (-1.125,1.5) circle (0.28125);
\filldraw[color=black!100, fill=yellow!50, very thick] (-1.125,0.75) circle (0.28125);
\filldraw[color=black!100, fill=green!50, very thick] (-0.375,0) circle (0.28125);
\filldraw[color=black!100, fill=purple!50, very thick] (0.375,0) circle (0.28125);
\filldraw[color=black!100, fill=brown!50, very thick] (1.125,0.75) circle (0.28125);
\filldraw[color=black!100, fill=blue!50, very thick] (1.125,1.5) circle (0.28125);

\node at (-1.125,1.5){\small $u_1$};
\node at (-1.125,0.75){\small $v_1$};
\node at (-0.375,0){\small $u_2$};
\node at (0.375,0){\small $v_2$};
\node at (1.125,0.75){\small $v_3$};
\node at (1.125,1.5){\small $u_3$};

\draw[black, very thick] (-4.425,1.5) -- (-4.425,0.75);
\draw[black, very thick] (-4.425,1.5) -- (-3.675,0);
\draw[black, very thick] (-4.425,1.5) -- (-2.925,0);
\draw[black, very thick] (-4.425,1.5) -- (-2.175,0.75);
\draw[black, very thick] (-4.425,1.5) -- (-2.175,1.5);
\draw[black, very thick] (-4.425,0.75) -- (-3.675,0);
\draw[black, very thick] (-4.425,0.75) -- (-2.925,0);
\draw[black, very thick] (-4.425,0.75) -- (-2.175,0.75);
\draw[black, very thick] (-4.425,0.75) -- (-2.175,1.5);
\draw[black, very thick] (-3.675,0) -- (-2.925,0);
\draw[black, very thick] (-3.675,0) -- (-2.175,0.75);
\draw[black, very thick] (-3.675,0) -- (-2.175,1.5);
\draw[black, very thick] (-2.925,0) -- (-2.175,0.75);
\draw[black, very thick] (-2.925,0) -- (-2.175,1.5);
\draw[black, very thick] (-2.175,0.75) -- (-2.175,1.5);

\filldraw[color=black!100, fill=green!50, very thick] (-4.425,1.5) circle (0.28125);
\filldraw[color=black!100, fill=yellow!50, very thick] (-4.425,0.75) circle (0.28125);
\filldraw[color=black!100, fill=blue!50, very thick] (-3.675,0) circle (0.28125);
\filldraw[color=black!100, fill=purple!50, very thick] (-2.925,0) circle (0.28125);
\filldraw[color=black!100, fill=brown!50, very thick] (-2.175,0.75) circle (0.28125);
\filldraw[color=black!100, fill=red!50, very thick] (-2.175,1.5) circle (0.28125);

\node at (-4.425,1.5){\small $u_1$};
\node at (-4.425,0.75){\small $v_1$};
\node at (-3.675,0){\small $u_2$};
\node at (-2.925,0){\small $v_2$};
\node at (-2.175,0.75){\small $v_3$};
\node at (-2.175,1.5){\small $u_3$};

\draw[black, very thick] (2.175,1.5) -- (2.175,0.75);
\draw[black, very thick] (2.175,1.5) -- (2.925,0);
\draw[black, very thick] (2.175,1.5) -- (3.675,0);
\draw[black, very thick] (2.175,1.5) -- (4.425,0.75);
\draw[black, very thick] (2.175,1.5) -- (4.425,1.5);
\draw[black, very thick] (2.175,0.75) -- (2.925,0);
\draw[black, very thick] (2.175,0.75) -- (3.675,0);
\draw[black, very thick] (2.175,0.75) -- (4.425,0.75);
\draw[black, very thick] (2.175,0.75) -- (4.425,1.5);
\draw[black, very thick] (2.925,0) -- (3.675,0);
\draw[black, very thick] (2.925,0) -- (4.425,0.75);
\draw[black, very thick] (2.925,0) -- (4.425,1.5);
\draw[black, very thick] (3.675,0) -- (4.425,0.75);
\draw[black, very thick] (3.675,0) -- (4.425,1.5);
\draw[black, very thick] (4.425,0.75) -- (4.425,1.5);

\filldraw[color=black!100, fill=blue!50, very thick] (2.175,1.5) circle (0.28125);
\filldraw[color=black!100, fill=purple!50, very thick] (2.175,0.75) circle (0.28125);
\filldraw[color=black!100, fill=green!50, very thick] (2.925,0) circle (0.28125);
\filldraw[color=black!100, fill=brown!50, very thick] (3.675,0) circle (0.28125);
\filldraw[color=black!100, fill=yellow!50, very thick] (4.425,0.75) circle (0.28125);
\filldraw[color=black!100, fill=red!50, very thick] (4.425,1.5) circle (0.28125);

\node at (2.175,1.5){\small $u_1$};
\node at (2.175,0.75){\small $v_1$};
\node at (2.925,0){\small $u_2$};
\node at (3.675,0){\small $v_2$};
\node at (4.425,0.75){\small $v_3$};
\node at (4.425,1.5){\small $u_3$};
\end{tikzpicture}
\end{center}

\begin{center}
\includegraphics[scale=0.25]{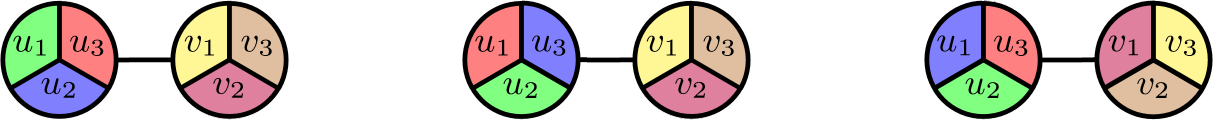}
\end{center}
\begin{center}
\textit{Fig. 2.3}
\end{center}
\end{samepage}
Indeed, while the three colorings on $G^b$ above are as distinct from one another, none
of their corresponding colorings on $G$, shown below each one of them, are distinct - in fact,
they only count as one coloring
\footnote{Recall that we only regard two colorings as distinct if the colors of the vertices are
different - the disposition of the colors among the sections is irrelevant.}.
Therefore, there cannot be the same number of distinct colorings
on the two graphs.
\begin{lemma}
Let $\mathcal{P}(G,\lambda,b)$ denote the number of distinct $b$-fold $\lambda$-colorings on the
graph $G$, or the \emph{fractional chromatic polynomial} of $G$. Then
\begin{equation}
\mathcal{P}(G,\lambda,b) = \frac{\mathcal{P}(G^b,\lambda,1)}{(b!)^{|V_G|}}
\end{equation}
\end{lemma}
\begin{proof}
Given any $b$-fold $\lambda$-coloring on $G$, if we permute the colors of the sections inside
an individual vertex - as one can see is the case in Figure 2.3 -, the resulting coloring on
$G$ is not distinct, but its corresponding coloring on $G^b$ is indeed distinct from the coloring
on $G^b$ corresponding to the initial coloring on $G$. For each individual $b$-fold $\lambda$-
coloring on $G$ there exist $(b!)^{|V_G|}$ such permutations, and thus there must be at least
$(b!)^{|V_G|}$ times more distinct $1$-fold $\lambda$-colorings on $G^b$ as there are $b$-fold
$\lambda$-colorings on $G$.\\
Conversely, given any two indistinct $b$-fold $\lambda$-coloring $f$ and $g$ on $G$ such that
their corresponding $1$-fold $\lambda$-colorings $f'$ and $g'$ on $G^b$ are distinct, they must
be permutations of one another, so there are exactly $(b!)^{|V_G|}$ times more distinct
$1$-fold $\lambda$-colorings on $G^b$ as there are $b$-fold $\lambda$-colorings on $G$.
\end{proof}
\pagebreak
\section{Complete graphs}
Before employing the above to derive a closed formula for the fractional chromatic polynomials
of complete graphs, we just need one more detail,
\begin{lemma}
If $G$ is a complete graph, then so is $G^b$.
\end{lemma}
\begin{proof}
This is an easy consequence of Definition 1. Suppose that $G$ is a
complete graph and $G^b$ is not. Then, there is some pair
of vertices $u$ and $v$ in $G^b$ with no edge between them.
But then, by the second part of Definition 1, there must be a corresponding
pair of vertices in $G$ with no edge between them. But this cannot be,
since $G$ is complete.
\end{proof}
More specifically,
\begin{equation}
K_{n}^{b} = K_{n \cdot b}
\end{equation}
And now we are equipped to find our closed formula.
We combine (2) and (3) and obtain
\begin{equation*}
\mathcal{P}(K_n,\lambda,b)=\frac{\mathcal{P}(K_{n \cdot b},\lambda,1)}{(b!)^n}
\end{equation*}
Then, we expand the numerator with the aid of (1) and get
\begin{equation*}
=\frac{\prod_{j=0}^{n \cdot b - 1} (\lambda - j)}{(b!)^n}
\end{equation*}
which can be simplified further, yielding
\begin{equation}
\mathcal{P}(K_n,\lambda,b) = \prod_{j=0}^{n-1} \binom{\lambda-bj}{b}
\end{equation}
\section{Trees and forests}
\begin{lemma}
If $T$ is a tree graph, then
\begin{equation}
\mathcal{P}(T,\lambda,b) = \binom{\lambda}{b}\binom{\lambda-b}{b}^{|V_T|-1}
\end{equation}
\end{lemma}
\begin{proof}
We already know that
\begin{equation*}
\mathcal{P}(T,\lambda,1) = \lambda(\lambda-1)^{|V_T|-1}
\end{equation*}
And there is an intuitive explanation behind this formula: first,
suppose $T$ is a path graph with vertices $u_1,\ldots,u_{|V_T|}$ - where $u_1$ and $u_{|V_T|}$ are end vertices -,
and we are given $T$ ``blank" and it is our task to $1$-fold $\lambda$-color $T$ properly. Due to the simple,
acyclic nature of path (and tree) graphs, we can do so with the following procedure: we begin by coloring some vertex of $T$,
$u_i$, chosen at random. Since all vertices are blank, we can choose any one of the $\lambda$ colors in the palette to color
$u_i$. We choose one of those and move on to color the adjacent vertex $u_{i+1}$. Now, to color $u_{i+1}$ we cannot use the same
color as the one we just did for $u_i$, so there are now only $\lambda-1$ colors for us to choose. Again, we choose one and
move on to color $u_{i+2}$. And again, to color $u_{i+2}$ we have only $\lambda-1$ colors to choose from - we
cannot use the same color we just did for $u_{i+1}$, but we can use the same color we did for $u_i$. And again we pick some
and move on to the next vertex, and so on, until we reach the end vertex $u_{|V_T|}$.\\
If $i=1$, \textit{i.e.} the initial vertex chosen was an end vertex, then through this procedure we have colored $T$ entirely.
Otherwise, we perform another run of the procedure: immediately after coloring $u_{|V_T|}$, we go color $u_{i-1}$ - for which
we will also have $\lambda-1$ choices of colors - and proceed headed to the other end vertex, $u_1$.
\medskip\\
For the initial vertex that we picked we had $\lambda$ choices of colors, and for all other vertices, $\lambda-1$ choices,
so in total there are $\lambda(\lambda-1)^{|V_T|-1}$ distinct $1$-fold $\lambda$-colorings on $T$. Naturally, this procedure
works for tree graphs which are not path graphs as well - at the end of each run, we start over by painting some blank vertex
adjacent to some colored vertex, until all vertices are colored - and indeed the formula above counts the number of distinct
$1$-fold $\lambda$-colorings on all tree graphs.
\medskip\\
Using the same procedure to count the number of $b$-fold $\lambda$-colorings ($\binom{\lambda}{b}$ choices on the initial
vertex and $\binom{\lambda-b}{b}$ choices on the remaining vertices) we then obtain (5).
\end{proof}
Indeed, (5) also gives us the number of distinct $b$-fold $\lambda$-colorings on individual components of a given forest graph,
\textit{i.e.}
\begin{equation*}
\binom{\lambda}{b}\binom{\lambda-b}{b}^{|c_1|-1}\binom{\lambda}{b}\binom{\lambda-b}{b}^{|c_2|-1}\cdots\binom{\lambda}{b}\binom{\lambda-b}{b}^{|c_k|-1}
\end{equation*}
is the number of distinct $b$-fold $\lambda$-colorings on a forest with $k$ components, where $|c_i|$ denotes the order of the $i$-th component.\\
Noting that in any forest $F$ with $k$ components, $|c_1|+|c_2|+\ldots+|c_k|=|V_F|$, we can simplify the above, ending up with
\begin{equation}
\mathcal{P}(F,\lambda,b)=\binom{\lambda}{b}^{k}\binom{\lambda-b}{b}^{|V_F|-k}
\end{equation}
\section{The Fundamental Reduction Theorem}
In Section 1, we mentioned that the recurrence relations of the Fundamental Reduction Theorem did not extend
to fractional colorings, but we did not demonstrate it. We will do so now.
\begin{proposition}
There exists at least one graph $G$, such that if $b>1$, then for some $\lambda$,
\begin{equation*}
\mathcal{P}(G,\lambda,b)\neq\mathcal{P}(G+uv,\lambda,b)+\mathcal{P}(G/uv,\lambda,b)
\end{equation*}
or 
\begin{equation*}
\mathcal{P}(G,\lambda,b)\neq\mathcal{P}(G-uv,\lambda,b)-\mathcal{P}(G/uv,\lambda,b)
\end{equation*} 
\end{proposition}
\begin{proof}
Indeed, the above is the case for most graphs, it actually being harder to find exceptions. As such, we provide only
a single example where (FRT 2) does not hold.\\
Let $G=K_3$. Then $G-uv=P_3$ and $G/uv=K_2$. We then plug (4) and (5) into (FRT 2) and obtain
\begin{equation*}
\prod_{j=0}^{2} \binom{\lambda-bj}{b} = \binom{\lambda}{b} \binom{\lambda-b}{b}^2 - \prod_{j=0}^{1} \binom{\lambda-bj}{b}
\end{equation*}
Expanding the products we arrive at
\begin{equation*}
\binom{\lambda}{b} \binom{\lambda-b}{b} \binom{\lambda-2b}{b} = \binom{\lambda}{b} \binom{\lambda-b}{b}^2 - \binom{\lambda}{b}\binom{\lambda-b}{b}
\end{equation*}
where dividing both sides by $\binom{\lambda}{b}\binom{\lambda-b}{b}$ yields
\begin{equation*}
\binom{\lambda-2b}{b} = \binom{\lambda-b}{b}-1
\end{equation*}
While the equation above clearly holds true for all $\lambda$ when $b=1$, it doesn't when $b>1$: for instance, if we plug in $b=2$ and $\lambda=6$,
we get $\binom{4}{2} = 2$, which is false.
\end{proof}
Now, to find a generalized form of (FRT 1) and (FRT 2), we can summon $G^b$ again to our aid. For instance, we know from (2)
that $\mathcal{P}(G^b,\lambda,1)=(b!)^{|V_G|}\cdot \mathcal{P}(G,\lambda,b)$, and we know from (FRT 2) that
\begin{equation*}
\mathcal{P}(G^b,\lambda,1)=\mathcal{P}(G^b-uv,\lambda,1)-\mathcal{P}(G^b/uv,\lambda,1)
\end{equation*}
Thus we combine the above with (2) and obtain
\begin{equation}
\mathcal{P}(G,\lambda,b)=\frac{\mathcal{P}(G^b-uv,\lambda,1)-\mathcal{P}(G^b/uv,\lambda,1)}{(b!)^{|V_G|}}
\end{equation}
And doing the same to (FRT 1) gives us
\begin{equation}
\mathcal{P}(G,\lambda,b)=\frac{\mathcal{P}(G^b+uv,\lambda,1)+\mathcal{P}(G^b/uv,\lambda,1)}{(b!)^{|V_G|}}
\end{equation}
the generalization of the Fundamental Reduction Theorem we were after.

\end{document}